\documentclass[graybox,final]{svmult}

\usepackage{mathptmx}       
\usepackage{helvet}         
\usepackage{courier}        
\usepackage{type1cm}        
\usepackage{makeidx}         
\usepackage{graphicx}        
\usepackage{multicol}        
\usepackage[bottom]{footmisc}
\usepackage{amsmath, amsfonts, amssymb}
\usepackage{color}
\usepackage{tikz}

\theoremstyle{plain}
\newtheorem{assumption}{Assumption}

\begin{document}

\title*{A hybrid discontinuous Galerkin method for transport equations on networks}
\titlerunning{A hybrid discontinuous Galerkin method for transport equations on networks}
\author{Herbert Egger and Nora Philippi}

\institute{
Herbert Egger
\at 
TU Darmstadt,\\
Karolinenplatz 5, 64289 Darmstadt, Germany\\
\email{egger@mathematik.tu-darmstadt.de}
\and 
Nora Philippi \at  TU Darmstadt,\\
Karolinenplatz 5, 64289 Darmstadt, Germany \\ 
\email{philippi@mathematik.tu-darmstadt.de}
}

\maketitle

\abstract{
We discuss the mathematical modeling and numerical discretization of transport problems on one-dimensional networks. 
Suitable coupling conditions are derived that guarantee conservation of mass across network junctions and dissipation of a  mathematical energy which allows to prove existence of unique solutions. We then consider the space discretization by a hybrid discontinuous Galerkin method which provides a suitable upwind mechanism to handle the transport problem and allows to incorporate the coupling conditions in a natural manner. In addition, the method inherits mass conservation and stability of the continuous problem. Order optimal convergence rates are established and illustrated by numerical tests.
\keywords{hybrid discontinuous Galerkin methods, transport problems, partial differential equations on networks
\\[5pt]
}
}

\section{Introduction}

Partial differential equations on networks arise in various applications including traffic flow, gas or water supply networks, and elastic multi-structures. We refer to \cite{GaravelloPiccoli,LagneseLeugeringSchmidt,Mugnolo}
for mathematical background, further applications, and references.
In this paper, we study scalar conservation laws on one dimensional network structures describing, e.g., the transport of a chemical substance in a flow through a network of pipes.
A linear advection equation is used to model the transport within the pipes and appropriate coupling conditions are formulated to describe the mixing of flows and the conservation of mass at network junctions. 
For the semi-discretization in space, we consider a hybrid discontinuous Galerkin method which turns out to be particularly well-suited for dealing with the hyperbolic nature of the problem and the coupling conditions at network junctions. Stability and conservation of the semi-discrete scheme as well as order optimal error estimates are established.

The rest of the paper is structured as follows: 
In section~\ref{sec:problem}, we introduce the basic notation and then give a complete formulation of the considered problem. A particular choice is made for the coupling conditions which allows us to prove conservation of mass and stability of the overall system. 
In section~\ref{sec:hdg}, we introduce the discretization and prove conservation, stability, and the error estimates. 
Some numerical tests are presented in section~\ref{sec:num} for illustration of our results.

\vspace*{-1em}

\section{Notation and problem formulation}
\label{sec:problem}

Following the notation of \cite{EggerKugler}, 
the topology of the pipe network is described by a finite, directed, and connected graph $\mathcal{G}=(\mathcal{V},\mathcal{E})$ with vertex set $\mathcal{V}=\{v_1,\dots,v_n\}$ and set of edges $\mathcal{E}=\{e_1,\dots,e_m\}\subset \mathcal{V}\times\mathcal{V}$.
For any vertex $v\in\mathcal{V}$, we denote by $\mathcal{E}(v)$ the set of edges  having $v$ as a vertex, and we distinguish between inner vertices (junctions) $\mathcal{V}_0=\{v\in\mathcal{V}:\vert\mathcal{E}(v)\vert\geq 2\}$ and boundary vertices $\mathcal{V}_{\partial}=\mathcal{V}\backslash\mathcal{V}_0$. 
For any edge $e = (v_i,v_j)$, we define $n^{e}(v_i)=-1$ and $n^{e}(v_j)=1$ to indicate the start and the end point of the edge, and we set $n^e(v)=0$ if $v \not\in e$. 
%
We further identify $e$ with the interval $[0,\ell^{e}]$ of positive length $\ell^e$ and denote by $L^2(e)=L^2(0,\ell^{e})$ the space of square integrable functions on the edge $e\in\mathcal{E}$, and by
\begin{equation*}
L^2(\mathcal{E})=L^2(e_1)\times\dots\times L^2(e_m)=\{u: u^{e}\in L^2(e)\ \text{for all}\ e\in\mathcal{E}\}
\end{equation*}
the corresponding space on the network. Here and below, $u^{e}=u\vert_e$ denotes the restriction of a function defined over the network to a single edge $e$. We use
\begin{equation*}
\Vert u\Vert_{L^2(\mathcal{E})}^2=\sum_{e\in\mathcal{E}}\Vert u^{e}\Vert_{L^2(e)}^2\qquad\text{and}\qquad (u,w)_{L^2(\mathcal{E})}=\sum_{e\in\mathcal{E}}(u^{e},w^{e})_{L^2(e)}
\end{equation*}
to denote the natural norm and scalar product of $L^2(\mathcal{E})$ and 
define by 
\begin{equation*}
H_{pw}^s(\mathcal{E})=\{u\in L^2(\mathcal{E}): u^{e}\in H^s(e)\ \text{for all}\ e\in\mathcal{E}\}
\end{equation*}
the broken Sobolev spaces which are equipped with the canonical norms
\begin{equation*}
\Vert u\Vert_{H_{pw}^s(\mathcal{E})}^2=\sum_{e \in \mathcal{E}} \|u^e\|_{H^s(e)}^2.
\end{equation*}
Let us note that $H^0_{pw}(\mathcal{E}) = L^2(\mathcal{E})$ and
for $s> 1/2$ the functions $u\in H^s_{pw}(\mathcal{E})$ are continuous along edges $e\in\mathcal{E}$ but may be discontinuous across junctions $v\in\mathcal{V}_0$.

\noindent
On every edge of the network, the transport is described by
\begin{alignat}{2}
	a^e(x)\partial_t u^e(x,t)+\partial_x (b^e u^e(x,t))=&\ 0, &&x\in e,\ t>0,\label{eq:sys1} \\
	u^e(x,0)=&\ u_0^e(x), \qquad &&x\in e.\label{eq:sys2}
\end{alignat}
Here $u^e$ is the concentration of the substance on pipe $e$, $a^e$ represents the cross sectional area of $e$, and $b^e$ is the given volume flow rate.
\begin{assumption} \label{ass:1} 
We have $a,\ b\in H^1_{pw}(\mathcal{E})$ with $a(x) \ge a_0 > 0$ and $b^e$ constant on every edge. Moreover, we require flow conservation at junctions, i.e., 
\begin{equation} \tag{C}
\sum_{e\in\mathcal{E}(v)}b^e n^e(v)=0\qquad \text{for all}\ v\in\mathcal{V}_0.\label{eq:sumbn} 
\end{equation}
\end{assumption}
The conditions on $b$ characterize an incompressible background flow. 
%
Using the above assumption, we can associate a unique flow direction to every edge and define for every vertex $v \in \mathcal{E}(v)$ the sets $\mathcal{E}^{\text{in}}(v)=\{e\in\mathcal{E}:b^{e}n^{e}(v)>0\}$ and $\mathcal{E}^{\text{out}}(v)=\{e\in\mathcal{E}:b^{e}n^{e}(v)<0\}$ of edges producing flow into or out of the vertex. 
We also decompose $\mathcal{V}_{\partial}$ into the sets of inflow and outflow vertices   $\mathcal{V}_{\partial}^{\text{in}}=\{v\in\mathcal{V}_{\partial}:b^{e}n^{e}(v)<0\ \text{for}\ e\in\mathcal{E}(v)\}$ and  $\mathcal{V}_{\partial}^{\text{out}}=\{v\in\mathcal{V}_{\partial}:b^{e}n^{e}(v)>0\ \text{for}\ e\in\mathcal{E}(v)\}$.
The local transport problems \eqref{eq:sys1}--\eqref{eq:sys2} are then complemented by coupling and boundary conditions
\begin{equation}\label{eq:sys3}
	u^e(v,t)=\hat{u}^v(t)\quad \text{for all } v\in\mathcal{V},\ e\in\mathcal{E}^{\text{out}}(v),\ t>0
\end{equation}
with auxiliary values $\hat{u}^v$ defined for $t \ge 0$ by the relations
\begin{equation}
	\hat{u}^v(t)=g^v(t),\qquad  v\in\mathcal{V}_{\partial}^{\text{in}} \label{eq:sys4}
\end{equation}
for inflow vertices. On the remaining vertices $v\in\mathcal{V}_0 \cup \mathcal{V}_\partial^{\text{out}}$, we set
\begin{equation}
\sum_{e\in\mathcal{E}^{\text{in}}(v)} b^e n^e(v)  \hat{u}^v(t) = \!\!\!\!\! \sum_{e\in\mathcal{E}^{\text{in}}(v)} b^e n^e(v)  u^e(v,t).\label{eq:sys5}
\end{equation}
From Assumption~\ref{ass:1}, we deduce that $\sum_{e\in \mathcal{E}^{\text{in}}(v)} b^e n^e(v)  > 0$, so that $\hat{u}^v$ is well-defined for all $v \in \mathcal{V}$ and can be eliminated using \eqref{eq:sys4} and \eqref{eq:sys5}. Furthermore, 
one can see that
$\hat{u}^v$ is a convex combination, i.e., a mixture, of the concentrations $u^e(v)$ in the flows entering the junction $v$. Using condition \eqref{eq:sumbn}, one can further see that the mass at inner vertices is conserved, more precisely
\begin{equation}
	\sum_{e\in\mathcal{E}^{\text{out}}(v)} b^e n^e(v)  \hat{u}^v(t) = -\!\!\!\!\! \sum_{e\in\mathcal{E}^{\text{in}}(v)} b^e n^e(v)  u^e(v,t)\quad \text{for all } v\in\mathcal{V}_0.\label{eq:sys6}
\end{equation}
The transport problem on networks is now fully described by \eqref{eq:sys1}--\eqref{eq:sys5}. 
It turns out that the number and type of coupling and boundary conditions is appropriate to guarantee stability and well-posedness of the problem and to ensure conservation of mass across network junctions. 
\begin{theorem}\label{WP-network}
Let Assumption~\ref{ass:1} hold. Then for any $u_0\in H_{pw}^1(\mathcal{E})$ and $g\in H^2(0,T;\mathcal{V}_\partial^{\text{in}})$ with $T>0$, satisfying \eqref{eq:sys3}--\eqref{eq:sys5} for $t=0$ with some $\hat{u}(0) = \hat{u}_0 \in \mathbb{R}^{|\mathcal{V}|}$,
the transport problem \eqref{eq:sys1}--\eqref{eq:sys5} has a unique solution
$u\in C^1([0,T];L^2(\mathcal{E}))\cap C^0([0,T];H_{pw}^1(\mathcal{E}))$ and $\hat u \in C^0([0,T];\mathbb{R}^{|\mathcal{V}|})$.
Moreover, the conservation property 
\begin{equation*} 
\frac{d}{dt}\int_{\mathcal{E}}au(x,t)\ dx=-\sum_{v\in\mathcal{V}_{\partial}} b^e n^e(v)  u^e(v,t)
\end{equation*}
holds as well as the energy identity
\begin{eqnarray*} 
\frac{d}{dt}\Vert a^{1/2}u \Vert^2_{L^2(\mathcal{E})} = 
&& -\sum_{v\in\mathcal{V}_\partial^{\text{out}}} |b^e n^e(v)| |u^{e}(v)|^2  +\sum_{v\in\mathcal{V}_\partial^{\text{in}}} |b^e n^e(v)| |g^v|^2\\
&& - \sum_{v \in \mathcal{V}_0} \sum_{e \in \mathcal{E}^{\text{in}}(v)} |b^e n^e(v)|  |u^e(v) - \hat{u}^v|^2. 
\end{eqnarray*}
\end{theorem}
\begin{proof}
The energy identity can be derived directly from \eqref{eq:sys1}--\eqref{eq:sys5} and establishes stability of the evolution problem. Existence of a unique solution then follows from the Lumer-Phillips theorem and semigroup theory \cite{Engel,Pazy}; a detailed proof can be found in \cite{Philippi}. 
Related results can also be found in \cite{Dorn,EggerKugler,Kramar,Mugnolo}. 
\end{proof}

\begin{remark}
For junctions with more than two inflow pipes, the last term in the energy estimated does in general not vanish and represents dissipation, i.e., loss of information, due to mixing.  
\end{remark} 

\vspace*{-1em}

\section{A hybrid discontinuous Galerkin method}
\label{sec:hdg}

We now formulate a discontinuous Galerkin method for the semi-discretization of problem~\eqref{eq:sys1}--\eqref{eq:sys5};
see \cite{Ern,Johnson} for a general introduction.
Hybridization introduces additional unknowns $\hat{u}^v$ at the grid points of the mesh which play a similar role as the auxiliary mixing values in the coupling conditions \eqref{eq:sys3}. 
The spatial grid is defined by
\begin{align*}
\mathcal{T}_h&=\{T^{e}_{i}=[x^{e}_{i-1},x^{e}_{i}]:i=1,\dots,M^{e},\ x^{e}_0=0,\ x^{e}_{M^{e}}=\ell^{e},\ e\in\mathcal{E}\}
\end{align*}
with local and global mesh size denoted by $h_i^{e}=x_{i}^{e}-x_{i-1}^{e}$ and $h=\max h_i^{e}$. 
As approximation spaces for the concentration field, we choose 
\begin{align*}
W_h&=\{w_h\in L^2(\mathcal{E}):w_h\vert_{T}\in P_k(T)\ \text{for all}\ T\in\mathcal{T}_h\},
\end{align*} 
i.e., spaces of piecewise polynomials of degree $ \le k$. 
Functions in $W_h$ may formally take multiple values at grid points. 
We introduce grid dependent scalar products
\begin{align*}
(u,w)_{\mathcal{T}_h}=\sum_{T\in\mathcal{T}_h}(u,w)_{L^2(T)},
\qquad 
\langle   u,w\rangle_{\partial\!\mathcal{T}_h}=\sum_{T\in\mathcal{T}_h} u(x_{i-1})w(x_{i-1})+u(x_{i})u(x_{i}), 
\end{align*}
where $T=[x_{i-1},x_i]$, and corresponding norms $\Vert w\Vert_{\mathcal{T}_h}^2=(w,w)_{\mathcal{T}_h}$ and $\vert w\vert_{\partial\!\mathcal{T}_h}^2=\langle w,w\rangle_{\partial\!\mathcal{T}_h}$. The broken Sobolev spaces over the mesh $\mathcal{T}_h$ are denoted by
\begin{align*}
H_{pw}^s(\mathcal{T}_h)=\{w\in L^2(\mathcal{E}):w\vert_T\in H^s(T)\ \text{for all}\ T \in\mathcal{T}_h\}.
\end{align*}
We further introduce the spaces of hybrid variables
\begin{equation*}
\hat{W}_h =\mathbb{R}^{\hat{M}}\quad\text{and}\quad \hat{W}_h^0=\{\hat{w}\in\hat{W}_h:\hat{w}^v=0\ \text{for all}\ v\in\mathcal{V}_{\partial}^{\text{in}}\}
\end{equation*}
with $\hat{M}=\vert\mathcal{V}\vert+\sum_{e\in\mathcal{E}}(M^e-1)$ denoting the total number of grid points. Note that grid points associated to the same junction $v \in \mathcal{V}$ are identified.
For the numerical approximation of \eqref{eq:sys1}--\eqref{eq:sys5}, we then consider the following semi-discrete scheme. 

\begin{problem}\label{prob:hdg} %
Find $u_h\in H^1([0,T];W_h)$ and $\hat{u}_h\in H^1([0,T];\hat{W}_h)$ such that $(u_h(0),w_h)_{\mathcal{T}_h}=(u_0,w_h)_{\mathcal{T}_h}$ for all $W_h$ and such that $\hat{u}_h^v(t) = g^v(t)$ for all $v\in\mathcal{V}_\partial^{\text{in}}$ and  
\begin{equation}\label{varform-HDG}
(a\partial_t u_h(t),w_h)_{\mathcal{T}_h}+b_h(u_h(t),\hat{u}_h(t);w_h,\hat{w}_h)=0
\end{equation}	
holds for all for all $w_h\in W_h$ and $\hat{w}_h\in\hat{W}_h^0$ and all $0 \le t \le T$, with bilinear form
\begin{equation}\label{def-bh}
b_h(u_h,\hat{u}_h;w_h,\hat{w}_h)=-(bu_h,\partial_xw_h)_{\mathcal{T}_h}+\langle bn \, u_h^\ast,w_h-\hat{w}_h\rangle_{\partial\!\mathcal{T}_h}+\langle bn\hat{u}_h(t),\hat{w}_h\rangle_{\mathcal{V}_{\partial}^{\text{out}}},
\end{equation} 
and upwind value $bn\, u_h^\ast=\max( b n,0) u_h + \min(bn,0) \hat{u}_h$ in flow direction.
\end{problem}
As noted in \cite{EggerSchoeberl}, the hybrid variable $\hat u_h$ can be eliminated from the system resulting in a standard discontinuous Galerkin discretization with upwind fluxes. At network junctions $v \in \mathcal{V}_0$, the hybrid variable $\hat u_h^v$ is determined by a discrete version of the coupling condition \eqref{eq:sys5}, which can be verified by appropriate testing. 
Let us start with summarizing some basic properties of the hybrid discontinuous Galerkin scheme. 
\begin{lemma}\label{lemma-bh-stab}
The bilinear form $b_h$ is semi-elliptic on the discrete spaces, i.e., \begin{equation*}
b_h(w_h,\hat{w}_h;w_h,\hat{w}_h)=\frac{1}{2}\big\vert \vert b\vert^{1/2}(w_h-\hat{w}_h)\big\vert_{\partial\!\mathcal{T}_h}^2+\frac{1}{2}\big\vert\vert b\vert^{1/2}\hat{w}_h\big\vert^2_{\mathcal{V}_{\partial^{\text{out}}}}
\ \ \forall w_h \in W_h, \ \hat{w}_h \in \hat{W}_h^0.
\end{equation*}
As a consequence, Problem~\ref{prob:hdg} is uniquely solvable. Moreover, the solution satisfies 
\begin{equation*} 
\frac{d}{dt}\int_{\mathcal{E}}au_h(x,t)\ dx=-\sum_{v\in\mathcal{V}_{\partial}} b^e n^e(v) u_h^e(v,t)
\end{equation*}
for all $0 \le t \le T$, as well as the discrete energy identity
\begin{equation*}
\frac{d}{dt}\Vert a^{1/2}u_h\Vert_{\mathcal{T}_h}^2+\big\vert \vert b\vert^{1/2}(u_h-\hat{u}_h)\big\vert^2_{\partial\!\mathcal{T}_h}+\vert \vert b\vert^{1/2}\hat{u}_h\vert_{\mathcal{V}_{\partial}^{\text{out}}}^2 = \big\vert \vert b\vert^{1/2}g\big\vert_{\mathcal{V}_{\partial}^{\text{in}}}^2.
\end{equation*}
Finally, let $(u,\hat{u})$ be a sufficiently regular solution of \eqref{eq:sys1}--\eqref{eq:sys5} and set $\hat{u}(x_i)=u(x_i)$ at grid points in the interior of the edges. Then 
\begin{equation*}
(a\partial_t u(t),w_h)_{\mathcal{T}_h}+b_h(u(t),\hat{u}(t);w_h,\hat{w}_h)=0 
\end{equation*}	 
for all $w_h \in W_h$, $\hat w_h \in \hat{W}^0_h$ and all $0 \le t \le T$,
i.e., the method is consistent. 
\end{lemma}
\begin{proof}
The semi-ellipticity of $b_h$ follows by standard arguments; see e.g. \cite{Ern,EggerKugler}. As a consequence of this identity and Assumption~\ref{ass:1}, $\hat{u}_h$ can be eliminated algebraically and the discrete problem can be turned into an linear ordinary differential equation. Existence of a unique solution then follows by the Picard-Lindel\"of theorem. The conservation property and the energy identity follow by appropriate testing. 
\end{proof}

\begin{remark}
The discretization inherits most of the properties from the continuous problem. The dissipation terms in the energy estimate are partly due to possible jumps across network junctions, which are present also on the continuous level, and partly due to jumps at interior vertices, which are caused by numerical dissipation due to the upwind mechanism in the discontinuous Galerkin method. 
\end{remark}
We are now in the position to establish order optimal a-priori error estimates.
\begin{theorem}\label{thm-conv-HDG}
Let $(u,\hat{u})$ denote a sufficiently regular solution of the system \eqref{eq:sys1}--\eqref{eq:sys5} and let $(u_h,\hat{u}_h)$ be the semi-discrete solution defined by Problem~\ref{prob:hdg}. Then
\begin{align*}
\Vert u-u_h\Vert_{L^\infty([0,T];L^2(\mathcal{E}))}\leq C_T h^{k+1}\vert u\vert_{W^{1,\infty}([0,T];H^{k+1}_{pw}(\mathcal{T}_h))},
\end{align*}
with constant $C_T$ only depending on the bounds for the coefficient $a$ and $T$.
\end{theorem}

\begin{proof}
As usual, the proof is based on an error splitting 
\begin{align*}
\Vert u-u_h\Vert_{L^\infty([0,T];L^2(\mathcal{E}))}\leq \Vert \eta_h \Vert_{L^\infty([0,T];L^2(\mathcal{E}))}+\Vert \epsilon_h \Vert_{L^\infty([0,T];L^2(\mathcal{E}))}
\end{align*}
into projection error $\eta_h = u - \pi_h u$ and discrete error $\epsilon_h = \pi_h u - u_h$. 
Similar to \cite{Thomee}, we use a particular projection $\pi_h:H^1_{pw}(\mathcal{E})\rightarrow W_h$ defined for element $T_i^e \in \mathcal{T}_h$ by 
\begin{equation*}
\pi_hw(x_{\text{i,out}}^{e})=w(x_\text{i,out}^{e}) 
\quad \text{and} \quad 
\int_{T_i^{e}}(w-\pi_hw)p\ dx=0 \quad \forall p\in\mathcal{P}_{k-1}(T_i^e).
\end{equation*}
Here $x_{i,out}^e$ is the outflow point of the element $T_i^e=[x_{i-1},x_i]$, i.e., $x_{i,out}^e = x_{i}^e$ if $b^e > 0$ and $x_{i,out}^e = x_{i-1}^e$ otherwise.
By standard estimates for this projection, we obtain 
\begin{align*}
\Vert \eta_h \Vert_{L^\infty([0,T];L^2(\mathcal{E}))}\leq Ch^{k+1}\Vert u\Vert_{L^\infty([0,T];H^{k+1}_{pw}(\mathcal{T}_h))}.
\end{align*}
Further define $\hat \pi_h u^v = \hat u^v$ for vertices $v \in \mathcal{V}$ of the network and $\hat \pi_h u(x_i^e) = u(x_i^e)$ for interior grid points $x_i^e$ on edge $e$. We abbreviate $\hat \epsilon_h = \hat \pi_h u - \hat u_h$, $\hat \eta_h = \hat u - \hat \pi_h u$, and denote by $\eta_h^\ast$ the upwind value as in the definition of the method. Note that $\hat \eta_h = 0$ and $\eta_h^\ast=0$ by construction. 
Using consistency of the discrete problem, we get
\begin{eqnarray*}
(a \partial_t \epsilon_h(t),w_h)_{\mathcal{T}_h} &+& b_h(\epsilon_h(t),\hat{\epsilon}_h(t);w_h,\hat{w}_h) \\
&=& (a \partial_t \eta_h(t),w_h)_{\mathcal{V}_h} + b_h(\eta_h(t),\hat{\eta}_h(t);w_h,\hat{w}_h)
\end{eqnarray*}
for all $w_h \in W_h$, $\hat w_h \in \hat W_h^0$, and $0 \le t \le T$.
Testing with $w_h = \epsilon_h$ and $\hat w_h = \hat \epsilon_h$ yields
\begin{align*}
\frac{1}{2}\frac{d}{dt}\Vert &a^{1/2} \epsilon_h\Vert_{\mathcal{T}_h}^2
=\underbrace{-b_h(\epsilon_h,\hat{\epsilon}_h;\epsilon_h,\hat{\epsilon}_h)}_{\leq 0}
  +(a\partial_t\eta_h,\epsilon_h)_{\mathcal{T}_h}
  +b_h(\eta_h,\hat{\eta}_h;\epsilon_h,\hat{\epsilon}_h)\\
&\leq
(a\partial_t\eta_h,\epsilon_h)_{\mathcal{T}_h}-\underbrace{(b\eta_h,\partial_x\epsilon_h)_{\mathcal{T}_h}}_{=0,\ \text{(proj.)}}
+\langle bn \, \underbrace{\eta_h^\ast}_{=0},\epsilon_h-\hat{\epsilon}_h\rangle_{\partial\!\mathcal{T}_h}+\langle bn\underbrace{\hat{\eta}_h}_{=0},\hat{\epsilon}_h\rangle_{\mathcal{V}_{\partial}^{\text{out}}}\\
&\leq \frac{c}{2} \Vert  \partial_t\eta_h\Vert_{\mathcal{T}_h}^2+\frac{1}{2}\Vert a^{1/2} \epsilon_h\Vert_{\mathcal{T}_h}^2.
\end{align*}
Let us note that the constant $c$ only depends on the bound for $a$ and the polynomial degree $k$. 
Integrating the remaining terms in time and applying Gronwall's lemma then allows to bound the discrete error by the projection error. 
\end{proof}

\begin{remark}
Using the semi-ellipticity of the discrete bilinear form, it is possible to obtain similar bounds also for the error $\hat \epsilon_h = \hat \pi_h u - \hat u_h = \hat u - \hat u_h$ at the grid points. 
A sub-sequent time discretization, e.g., by implicit Runge-Kutta methods, can also be analyzed with standard arguments; see \cite{Ern,Thomee}. Since the problem is one-dimensional, 
the computational overhead of an implicit time integration scheme is negligible.
\end{remark}

\vspace{-1em}

\section{Numerical tests}
\label{sec:num}

For our numerical tests, we consider the following network topology.
\begin{center}
\begin{tikzpicture}[thick,scale=1.5, every node/.style={scale=0.8}]
\node (A) at (0,0) [circle,draw,thick] {$v_1$};
\node (B) at (1,0) [circle,draw,thick] {$v_2$};
\node (C) at (1.75,0.5) [circle,draw,thick] {$v_3$};
\node (D) at (1.75,-0.5) [circle,draw,thick] {$v_4$};
\node (E) at (2.5,0) [circle,draw,thick] {$v_5$};
\node (F) at (3.5,0) [circle,draw,thick] {$v_6$};

\draw[->, thick] (A) to node[above] {$e_1$} (B);
\draw[->, thick] (B) to node[above] {$e_2$} (C);
\draw[->, thick] (B) to node[above] {$e_3$} (D);
\draw[->, thick] (C) to node[right] {$e_4$} (D);
\draw[->, thick] (C) to node[above] {$e_5$} (E);
\draw[->, thick] (D) to node[above] {$e_6$} (E);
\draw[->, thick] (E) to node[above] {$e_7$} (F);
\end{tikzpicture}
\end{center}
We set $\ell^{e}=1$ and $a^{e}=1$ for all edges, and choose $b^{e_1}=2,\ b^{e_2}=b^{e_3}=1$, 
$b^{e_4}=b^{e_5}=0.5$, $b^{e_6}=1.5$, and $b^{e_7}=2$ which satisfies condition \eqref{eq:sumbn}. 
We further choose $u_0^{e}= 0$ as initial conditions and $g^{v_1}(t)=t^2/25$ as inflow boundary conditions, which fulfill the compatibility condition $u_0^{e_1}(0)=g^{v_1}(0)$. 
The solution for this problem can be computed analytically and one can verify that $u\in W^{1,\infty}(0,T;H_{pw}^2(\mathcal{T}_h))$. From the estimates of Theorem~\ref{thm-conv-HDG}, we therefore expect second order convergence when discretizing with piecewise polynomials of order $k=1$. 
For time integration, we utilize an implicit Euler method with sufficiently small step size $\tau \le h^2$, and we use
\begin{equation*}
err = \max_{0\leq t^n\leq T}\Vert I_h u(t_n)-u_h^n\Vert_{L^2(\mathcal{E})}
\end{equation*}
as measure for the error, where $I_h u$ denotes the element-wise linear interpolation.
\begin{figure}[ht]
  \begin{center}
    \begin{minipage}[b]{0.6\textwidth}
      \includegraphics[scale=0.4]{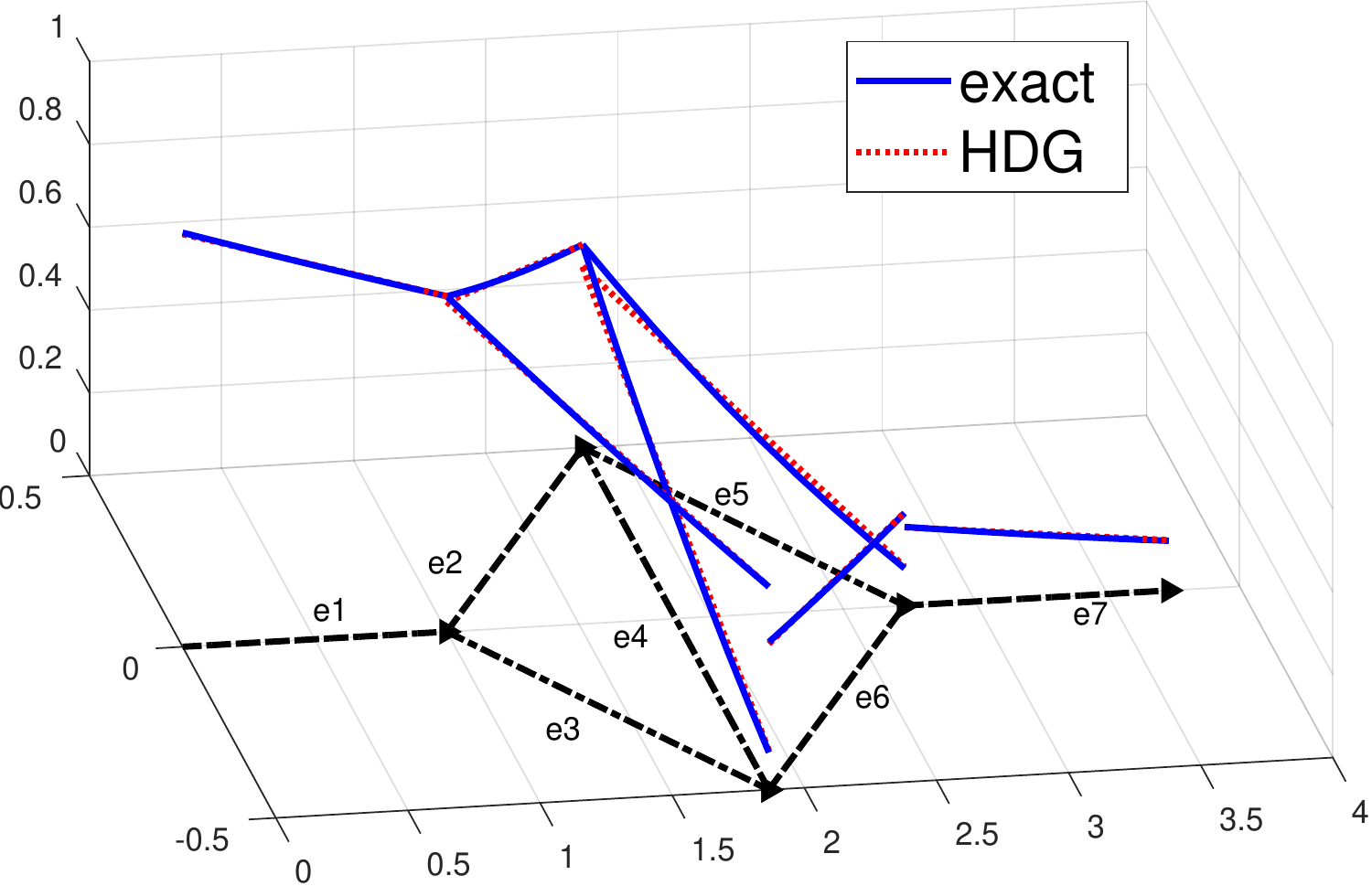}
    \end{minipage}
    \begin{minipage}[b]{0.35\textwidth}
      \centering
      \bgroup
      \setlength\tabcolsep{1em}
      \begin{tabular}{l||c|c}
        \rule{0mm}{2.3ex} $h$ & err & rate\\
        \hline
        \rule{0mm}{2.3ex} $2^0$    &  0.0303 &  --- \\
        \rule{0mm}{2.3ex} $2^{-1}$ &  0.0076 &  1.9948\\
        \rule{0mm}{2.3ex} $2^{-2}$ &  0.0019 &  2.0010\\
        \rule{0mm}{2.3ex} $2^{-3}$ &  0.0005 &  1.9774\\
        \rule{0mm}{2.3ex} $2^{-4}$ &  0.0001 &  1.9978\\
        \rule{0mm}{2.3ex} $2^{-5}$ &  0.0000 &  1.9725\\
      \end{tabular}
      \egroup
      \vspace{5ex}
    \end{minipage}
  \end{center}
  \caption{Left: Snapshot of the exact solution $u$ (blue) and the hybrid dG solution $u_h$ for mesh size $h=1$ (red, dashed). The discontinuity at network junctions is clearly visible.  Right: Error and convergence rates for time horizon $T=5$. As expected we observe second order convergence.}
  \label{fig}
\end{figure} 
As expected, the convergence rates observed in our numerical tests coincide with predictions from Theorem~\ref{thm-conv-HDG}. The solution plot in Figure~\ref{fig} clearly illustrates the discontinuity of the analytical solution at network junctions.

\begin{acknowledgement}
This work was supported by the German Research Foundation (DFG) via grants TRR~154 C4 and the ``Center for Computational Engineering'' and at TU Darmstadt. 
\end{acknowledgement}


\begin{thebibliography}{10}
\providecommand{\url}[1]{{#1}}
\providecommand{\urlprefix}{URL }
\expandafter\ifx\csname urlstyle\endcsname\relax
  \providecommand{\doi}[1]{DOI~\discretionary{}{}{}#1}\else
  \providecommand{\doi}{DOI~\discretionary{}{}{}\begingroup
  \urlstyle{rm}\Url}\fi

\bibitem{Ern}
Di~Pietro, D. A., Ern, A.: Mathematical aspects of discontinuous Galerkin
  methods.
\newblock Springer Science \& Business Media (2011)

\bibitem{Dorn}
Dorn, B.: Semigroups for flows on infinite networks.
\newblock M.Sc. thesis, Eberhard Karls Universit\"at T\"ubingen (2005)

\bibitem{EggerKugler}
Egger, H., Kugler, T.: Damped wave systems on networks: Exponential stability and uniform approximations.
\newblock Numer. Math. 138, 839--867 (2018)

\bibitem{EggerSchoeberl}
Egger, H., Sch\"oberl, J.: A hybrid mixed discontinuous {G}alerkin finite element method for convection-diffusion problems.
\newblock IMA J. Num. Anal. 30, 1206--1234 (2009)

\bibitem{Engel}
Engel, K. J., Nagel, R.: One-Parameter Semigroups for Linear Evolution
  Equations, 1 edn.
\newblock Springer-Verlag New York (2000)

\bibitem{GaravelloPiccoli}
Garavello, M., Piccoli, B.: Traffic flow on networks, \emph{AIMS Series on Applied Mathematics}, vol.~1.
\newblock American Institute of Mathematical Sciences (AIMS), Springfield, MO (2006)

\bibitem{Johnson}
Johnson, C.: Numerical Solution of Partial Differential Equations by the Finite Element Method.
\newblock Dover Publications (2009)

\bibitem{Kramar}
Kramar, M., Sikolya, E.: Spectral properties and asymptotic periodicity of flows in networks.
\newblock Mathematische Zeitschrift 249, 139--162 (2005)

\bibitem{LagneseLeugeringSchmidt}
Lagnese, L. E., Leugering, G., Schmidt, E. J. P. G.: Modeling, Analysis and Control of Dynamic Elastic Multi-Link Structures.
\newblock Systems \& Control: Foundations \& Applications. Springer
  Science+Business Media, New~York (1994)

\bibitem{Mugnolo}
Mugnolo, D.: Semigroup methods for evolution equations on networks.
\newblock Springer (2014)

\bibitem{Pazy}
Pazy, A.: Semigroups of Linear Operators and Applications to Partial
  Differential Equations, 1 edn.
\newblock Springer-Verlag New York (1983)

\bibitem{Philippi}
Philippi, N.: Analysis and numerical approximation of transport equations on networks.
\newblock M.Sc. thesis, TU Darmstadt (2019)

\bibitem{Thomee}
Thom{\'e}e, V.: Galerkin finite element methods for parabolic problems. \newblock Springer (1984)

\end{thebibliography}

\end{document}